\documentclass[11pt]{amsart}
\usepackage{amscd,amssymb}

\input xy
\xyoption{all}

\newcommand{\IR}{\mathbb R}
\newcommand{\IQ}{\mathbb Q}

\newcommand{\IZ}{\mathbb Z}
\newcommand{\IS}{\bar{\mathbb R}}
\newcommand{\IK}{\mathbb K}
\newcommand{\IF}{\mathbb F}
\newcommand{\U}{\mathcal U}
\newcommand{\F}{\mathcal F}
\newcommand{\pr}{\mathrm{pr}}
\newcommand{\dom}{\mathrm{dom}}

\newcommand{\edim}{\mbox{\rm e-dim}\,}
\newcommand{\w}{\omega}

\newcommand{\id}{\mathrm{id}}
\newcommand{\GGamma}{\bar\Gamma}
\newcommand{\EE}{\mathcal E}

\newtheorem{theorem}{Theorem}[section]

\newtheorem{proposition}[theorem]{Proposition}
\newtheorem{corollary}[theorem]{Corollary}
\newtheorem{conjecture}[theorem]{Conjecture}

\newtheorem{problem}[theorem]{Problem}
\newtheorem{lemma}[theorem]{Lemma}
\theoremstyle{definition}
\newtheorem{remark}[theorem]{Remark}

\title[The dimension of the space of $\mathbb R$-places of certain rational function fields]{The dimension of the space of $\mathbb R$-places\\ of certain rational function fields}
\author{T.Banakh, Ya.Kholyavka, K.Kuhlmann, M.Machura, O.Potyatynyk}
\begin{document}

\begin{abstract} We prove that the space $M(K(x,y))$ of $\mathbb R$-places of the field $K(x,y)$ of rational functions of two variables
with coefficients in a totally Archimedean field $K$ has covering and integral dimensions $\dim M(K(x,y))=\dim_\IZ M(K(x,y))=2$
and the cohomological dimension $\dim_G M(K(x,y))=1$ for any Abelian 2-divisible coefficient group $G$.
\end{abstract}
\address{T.~Banakh: Department of Mathematics, Ivan Franko National University of Lviv, Ukraine, and\newline Instytut Matematyki, Jan Kochanowski University, Kielce, Poland}
\email{t.o.banakh@yahoo.com}
\address{Ya.~Kholyavka and O.Potyatynyk:  Department of Mathematics, Ivan Franko National University of Lviv, Ukraine}
\email{ya\_khol@franko.lviv.ua, oles2008@gmail.com}
\address{M.~Machura and K.~Kuhlmann: Instytut Matematyki, University of Silesia, Katowice, Poland}
\email{machura@ux2.math.us.edu.pl, kmk@math.us.edu.pl}
\maketitle

\section{Introduction}
In this paper we study the topological structure and evaluate the topological dimensions of the spaces of $\mathbb R$-places of a field $K$ and of its transcendental extensions $K(x_1,\dots,x_n)$ consisting of rational functions of $n$ variables with coefficients in the field $K$.

The shortest possible way to introduce {\em $\mathbb R$-places} on a field $K$ is to define them as functions $\chi:K\to\IS=\IR\cup\{\infty\}$ to the extended real line, preserving the arithmetic operations in the sense that $\chi(0)=0$, $\chi(1)=1$,
$\chi(x+y)\in \chi(x)\oplus \chi(y)$ and $\chi(x\cdot y)\in \chi(x)\odot \chi(y)$ for all $x,y\in K$, where $\oplus$ and $\odot$ are multivalued extensions of the addition and multiplication operations from $\IR$ to $\IS$.
By definition, for $r,s \in \IS$,  $r\oplus s=\{r+s\}$ if $r+s\in\IS$ is defined and $r\oplus s=\IS$ if $r+s$ is not defined, which happens if and only if $r=s=\infty$, in which case $\infty\oplus\infty=\IS$.  By analogy, we define $r\odot s$: it equals the singleton $\{r\cdot s\}$ if $r\cdot s$ is defined and $\IS$ in the other case, which happens if and only if $\{r,s\}=\{0,\infty\}$.

Historically, $\mathbb R$-places appeared from studying ordered fields.
By an {\em ordered field} we understand a pair $(K,P)$ consisting of a field $K$ and a subset $P\subset K$ called {\em the positive cone} of
$(K,P)$ such that $P$ is an additively closed subgroup of index 2 of the  multiplicative group of $K$.
%\begin{itemize}
%\item $x+y,xy\in P$ for all $x,y\in P$;
%\item $x^2\in P$ for all $x\in K\setminus\{0\}$;
%\item $-1\notin P$;
%\item $K=-P\dot\cup\{0\}\dot\cup P$.
%\end{itemize}
There is a bijective correspondence between positive cones of $K$ and linear orders compatible with addition and multiplication by positive elements.
The set $\{a \in K: a>0\}$ is a positive cone, and the positive cone $P$
generates a total order $<$ on $K$ defined by $x<y \Leftrightarrow y-x \in P$.
Each ordered field $(K,P)$ has characteristic zero and hence contains the field $\IQ$ of rational numbers as a subfield. This fact allows us to define the {\em Archimedean part}
$$A_P(K)=\{x\in K: \exists a,b\in\IQ:a<x<b\}$$ of the ordered field $(K,P)$ and also to define the canonical $\mathbb R$-place $\chi_P:K\to\bar\IR$ on $K$ assigning $\chi_P(x)=\infty$ to each $x\in K\setminus A_P(K)$ and $$\chi_P(x)=\sup\{a\in \IQ:a\le x\}=\inf\{b\in\IQ:b\ge x\}\in\IR$$ to each $x\in A_P(K)$. Here the supremum and infimum is taken in the ordered field $\IR$ of real numbers.

According to Theorems 1 and 6 of \cite{Lang53}, a field $K$ admits a $\mathbb R$-place if and only if it is {\em orderable} in the sense that it admits a total order. By \cite{Brown}, each $\mathbb R$-place $\chi:K\to\bar\IR$ on a field $K$ is generated by a suitable total order $P$ on $K$.

For an orderable field $K$ denote by $\mathcal X(K)$ the space of total orders on $K$ and by $M(K)$ the space of $\mathbb R$-places on $K$. The mentioned results \cite{Lang53} and \cite{Brown} imply that the map
 $$\lambda:\mathcal X(K)\to M(K),\;\;\lambda:P\mapsto\chi_P,$$ assigning to each total order $P$ on $K$ the corresponding $\mathbb R$-place $\chi_P$ is surjective.

The spaces $\mathcal X(K)$ and $M(K)$ carry natural compact Hausdorff topologies.
Namely, $\mathcal X(K)$ carries the Harrison topology generated by the subbase consisting of the sets $a^+=\{P\in\mathcal X(K):a\in P\}$ where $a\in K\setminus\{0\}$. According to \cite[6.1]{Dubois}, the space $\mathcal X(K)$ endowed with the Harrison topology is compact Hausdorff and zero-dimensional. By \cite{Craven}, each compact Hausdorff zero-dimensional space is homeomorphic to the space of orderings $\mathcal X(K)$ of some field $K$.

To introduce a natural topology on the space $M(K)$ of $\mathbb R$-places of a field $K$, first endow the extended real line $\bar \IR=\IR\cup\{\infty\}$ with the topology of
one-point compactification of the real line $\IR$. It follows from the definition of $\IR$-places that the space $M(K)$ is a closed subspace of
the compact Hausdorff space $\bar\IR^K$ of all functions from $K$ to $\bar\IR$, endowed with the topology of Tychonoff product of the circles $\bar\IR$.  So, $M(K)$ is a compact Hausdorff space, being a closed subspace of the compact Hausdorff space $\bar\IR^K$.

It turns out that the topology induced on $M(K)$ by the product topology coincides with the quotient topology induced by the mapping
$\lambda:\mathcal X(K)\to M(K)$. This can be seen as follows. By \cite{Kuhlmann}, the sets
$$U(a)=\{ \chi \in M(K): \chi(a) \in (0,\infty)\}, \,\,\, a\in K,$$
compose a sub-basis of the quotient topology on $M(K)$. Since those sets are open in the product topology of $M(K)$, the quotient topology is weaker than the product topology. Since the quotient topology is Hausdorff (see \cite[Cor.9.9]{Lam}) and
the product topology is compact (so the weakest among Hausdorff topologies), both topologies on $M(K)$ coincide.

The space $M(K)\subset\bar\IR^K$ is metrizable if the field $K$ is countable. The converse statement is not true as the uncountable field $\IR$ has trivial space of $\mathbb R$-places $M(\IR)=\{\id\}$.
The space of $\mathbb R$-places $M(\IR(x))$ of the field $\IR(x)$ is homeomorphic to the projective line $\bar\IR$ while $M(\IR(x,y))$ is not metrizable, see \cite{MMO}.

In this paper we shall address the following general problem posed in \cite{BG}.

\begin{problem} Investigate the interplay between algebraic properties of a field $K$ and topological properties of its space of $\mathbb R$-places $M(K)$.
\end{problem}

We shall be mainly interested in the fields $K(x_1,\ldots,x_n)$ of rational functions of $n$ variables with coefficients in a subfield $K\subset \IR$. It is known that a field $K$ is isomorphic to a subfield of $\IR$ if and only if $K$ admits an Archimedean order, i.e., a total ordering $P$ whose Archimedean part $A_P(K)$ coincides with $K$. This happens if and only if the corresponding $\mathbb R$-place $\chi_P:K\to\bar\IR$ is injective if and only if $\chi_P(K)\subset\IR$. By $M_A(K)$ we denote the space of injective $\mathbb R$-places on $K$. Observe that $M_A(K)$ coincides with the space of homomorphisms from $K$ to the real line $\IR$.

A field $K$ will be called {\em totally Archimedean} if it is orderable and each total order on $K$ is Archimedean.
Such fields were introduced and characterized in \cite{Archimed}. Important examples of totally Archimedean fields are the fields $\IQ$ and $\IR$.
For a totally Archimedean field $K$ the quotient map $\lambda:\mathcal X(K)\to M(K)$ is injective. In this case, the spaces $M(K)$ and $\mathcal X(K)$ are homeomorphic and hence the space $M(K)$ is zero-dimensional.

In this paper we shall attack the following:

\begin{conjecture}\label{con1.2} For any subfield $K\subset\IR$ and every natural number $n$ the space $M(K(x_1,\dots,x_n)$ of $\mathbb R$-places
 of the field $K(x_1,\dots,x_n)$ has topological dimension $\dim(K(x_1,\dots,x_n))\ge n$. If the field $K$ is
totally Archimedean, then $\dim(K(x_1,\dots,x_n))=n$.
\end{conjecture}

For $n=1$ this conjecture was confirmed (in a stronger form) in \cite{Kuhlmann}: $\dim M(K(x))=1$ for any (also non-Archimedean) real closed field $K$. The main result of this paper is the following theorem confirming Conjecture~\ref{con1.2} for $n\le 2$.

\begin{theorem}\label{t1} For any field $K$ admitting an Archimedean order, we get
$\dim M(K(x))\ge 1$ and $\dim M(K(x,y))\ge 2$. If the field $K$ is totally Archimedean, then
$\dim M(K(x))=1$ and $\dim M(K(x,y))=2$.
\end{theorem}

Actually, Theorem~\ref{t1} does not say all the truth about the dimension of the space $M(K(x,y))$.
It turns out that this space has covering topological dimension 2 but for any 2-divisible group $G$ the cohomological
dimension $\dim_G M(K(x,y))$ is equal to 1! So, the space $M(K(x,y))$ is a natural example of a compact space that is not
dimensionally full-valued (which means that the cohomological dimensions of $M(K(x,y))$ for various coefficient groups $G$ do not coincide).
A classical example of such a space is the Pontryagin surface, that is a surface with M\"obius bands glued at each point of
a countable dense subset, see \cite[1.9]{Dran}.

The covering and cohomological dimensions are partial cases of the extension dimension defined as follows, see \cite{DD}. We say that the {\em extension dimension} of a topological space $X$ does not exceed a topological space $Y$ and write $\edim(X)\le Y$ if each continuous map $f:A\to Y$ defined on a closed subspace $A$ of $X$ can be extended to a continuous map  $\bar f:X\to Y$. The classical Hurewicz-Wallman characterization of the covering dimension \cite[1.9.3]{End} says that $\dim(X)\le n$ for a separable metric space $X$ if and only if $\edim(X)\le S^n$ where $S^n$ stands for the $n$-dimensional sphere. The sphere $S^n$ is an example of a Moore space $M(\IZ,n)$ (whose reduced homology groups $\tilde H_k(S^n)$, $k\ne n$, are trivial except for the $n$-th group $\tilde H_n(S^n)$ which is isomorphic to $\IZ$).

For a non-trivial abelian group $G$ the {\em cohomological dimension} $\dim_G(X)$ of a compact space $X$ coincides with the smallest
 non-negative number $n$ such that $\edim(X)\le K(G,n)$ where $K(G,n)$ is the Eilenberg-MacLane complex of $G$
(this is a CW-complex having all homotopy groups trivial except for the $n$-th homotopy group $\pi_n(K(G,n))$ which is isomorphic to $G$).
 If no such $n$ exists, then we put $\dim_G(X)=\infty$.   It is known that $\dim_G(X)\le\dim(X)$ for each abelian group $G$
and $\dim(X)=\dim_\IZ(X)$ for any finite-dimensional compact space $X$. On the other hand, the famous Pontryagin
surface $\Pi_2$ has covering dimension $\dim \Pi_2=2$ and cohomological dimension  $\dim_G(\Pi_2)=1$ for any 2-divisible
abelian group $G$, see \cite[1.9]{Dran}. A group $G$ is called {\em 2-divisible} if for each $x\in G$ there is  $y\in G$ with $y^2=x$.
Surprisingly, but for any totally Archimedean field $K$ the space $M(K(x,y))$ has the same pathological dimension properties:

\begin{theorem}\label{t2} For any totally Archimedean field $K$ the space of $\mathbb R$-places $M(K(x,y))$ has integral cohomological dimension $\dim_\IZ M(K(x,y))=\dim M(K(x,y))=2$ and the cohomological dimension $\dim_G M(K(x,y))=1$ for any non-trivial 2-divisible Abelian group $G$.
\end{theorem}

Theorems~\ref{t1} and \ref{t2} will be proved in Section~\ref{s4} after some preliminary work made in Section~\ref{s3}.

\section{Graphoids and spaces of $\mathbb R$-places}\label{s3}

In this section we shall discuss the interplay between spaces of $\mathbb R$-places and graphoids.
The notion of a graphoid has topological nature and can be defined for any family $\F$ of partial functions between topological spaces.

By a {\em partial  function} between topological spaces $X,Y$ we understand a continuous function $f:\dom(f)\to Y$ defined on a subspace $\dom(f)$ of the space $X$.
Its {\em graphoid} $\GGamma(f)$ is the closure of its graph $\Gamma(f)=\{(x,f(x)):x\in\dom(f)\}$ in the Cartesian product $X\times Y$. The graphoid $\GGamma(f)$ determines a multi-valued extension $\bar f:X\multimap Y$ of $f$ whose graph $\Gamma(\bar f)=\{(x,y)\in X\times Y:y\in\bar f(x)\}$ coincides with the graphoid $\GGamma(f)$ of $f$. The multivalued function $\bar f:X\multimap Y$ assigns to each point $x\in X$ the closed subset $\bar f(x)=\{y\in Y:(x,y)\in\GGamma(f)\}$ of the space $Y$.

For a finite family $\F$ of partial functions between topological spaces $X,Y$ we define the graphoid $\GGamma(\F)$ of $\F$ as the graphoid of the ``vector'' function $$\F:\dom(\F)\to Y^\F,\;\;\F:x\mapsto (f(x))_{f\in\F},$$defined on the subset $\dom(\F)=\bigcap_{f\in \F}\dom(f)$.

For an arbitrary family $\F$ of partial functions between $X$ and $Y$ we define its graphoid $\GGamma(\F)$ as the intersection
$$\GGamma(\F)=\bigcap\{\pr_\EE^{-1}(\GGamma(\EE)):\EE\subset\F,\;\;|\EE|<\infty\}\subset X\times Y^\F$$
where for $\EE\subset\F$
$$\pr_\EE:X\times Y^\F\to X\times Y^\EE,\;\;\pr_\EE:(x,(y_f)_{f\in\F})\mapsto (x,(y_f)_{f\in\EE}),$$denotes the natural projection.

The following lemma describing the structure of the graphoid $\GGamma(\F)$ easily follows from the definition of $\GGamma(\F)$.

\begin{lemma}\label{l3.1} The graphoid $\GGamma(\F)$ consists of all points $(x,(y_f)_{f\in\F})\in X\times Y^\F$ such that for any finite subfamily $\EE\subset\F$ and neighborhoods $O(x)\subset X$ and $O(y_f)\subset Y$ of the points $x$ and $y_f$, $f\in\EE$, there is a point $x'\in O(x)\cap\dom(\EE)$ such that $f(x')\in O(y_f)$ for all $f\in\EE$.
\end{lemma}

Now we consider the graphoids in the context of rational functions of $n$ variables. To shorten notation, we shall denote the $n$-tuple $(x_1,\dots,x_n)$ by $\vec x$. So, $K(\vec x)$ will denote the field $K(x_1,\dots,x_n)$ of rational functions of $n$ variables with coefficients in a field $K$.

Observe that each rational function $f\in\IR(\vec x)$, written as an irreducible fraction $f=\frac{p}{q}$ of two polynomials $p,q\in\IR(\vec x)$, can be thought as a partial function $f:\dom(f)\to \bar\IR$ defined on the open dense subset $\dom(f)=\IR^n\setminus (p^{-1}(0)\cap q^{-1}(0))$ of the $n$-dimensional torus $\bar\IR^n$.

Now we see that any family of rational functions $\F\subset\IR(\vec x)$ can be considered as a family of partial functions whose graphoid $\GGamma(\F)\subset\bar\IR^n\times\bar\IR^\F$ is a well-defined closed subset of the compact Hausdorff space $\bar\IR^n\times\bar\IR^\F$.

Observe that for any finite subfamily $\F\subset\IR(\vec x)$ the subset $\dom(\F)=\bigcap_{f\in \F}\dom(f)$ is open and dense in $\IR^n$.
Thus the graphoid $\GGamma(\F)\subset\bar\IR^n\times\bar\IR^\F$ projects surjectively onto the $n$-torus $\bar\IR^n$.
The same fact is true for any family $\F\subset\IR(\vec x)$: its graphoid $\GGamma(\F)$ projects surjectively onto the $n$-torus $\bar\IR^n$.

It turns out that for a subfield $\F\subset\IR(\vec x)$, containing $\IQ(\vec x)$, the graphoid $\GGamma(\F)$ can be
identified with a subspace of the space of $\mathbb R$-places $M(\F)$.

\begin{theorem}\label{t3.2} Let $\F\supset\IQ(\vec x)$ be a subfield of the field $\IR(\vec x)$.
\begin{enumerate}
\item Each point $\gamma=\big(\vec a,(y_f)_{f\in\F}\big)$ of the graphoid $\GGamma(\F)\subset\bar\IR^n\times\bar\IR^\F$ determines an $\mathbb R$-place
$$\delta_\gamma:\F\to\bar\IR,\;\;\delta_\gamma:f\mapsto y_f.$$
To each rational function $f\in\F$ this $\mathbb R$-place assigns a point $\delta_\gamma(f)\in\bar f(\vec a)$ where $\bar f:\bar\IR^n\multimap\bar\IR$ is the multivalued extension of $f$ whose graph $\Gamma(\bar f)$ coincides with the graphoid $\GGamma(f)$ of $f$.
\item The map $$\delta:\GGamma(\F)\to M(\F),\;\;\delta:\gamma\mapsto \delta_\gamma,$$is a topological embedding.
\item If $\F=\IK(\vec x)$ for some subfield $\IK\subset\IR$, then $$\delta(\GGamma(\F))=\{\chi\in M(\F):\chi|\IK=\id\}.$$
\end{enumerate}
\end{theorem}

\begin{proof} 1. Fix a point  $\gamma=\big(\vec a,(b_f)_{f\in\F}\big)\in\GGamma(\F)\subset\bar\IR^n\times\bar\IR^\F$ and consider the function $\delta_\gamma:\F\to\bar\IR$, $\delta_\gamma:f\mapsto b_f$.

Given any rational function $f\in\F$ consider its graphoid $\GGamma(f)$, which is equal to the closure of its graph $\{(\vec x,f(\vec x)):\vec x\in\dom(f)\}$ in $\bar\IR^n\times\bar\IR$. Next, consider the projection $$\pr_f:\bar\IR^n\times\bar\IR^\F\to\bar\IR^n\times\bar\IR,\;\; \pr_f:\big(\vec x,(y_f)_{f\in\F}\big)\mapsto\big(\vec x,y_f\big)$$ and observe that
$\pr_f(\gamma)=\big(\vec a,b_f\big)\in\GGamma(f)=\Gamma(\bar f)$ by the definition of the graphoid $\GGamma(\F)$.
Consequently, $\delta_{\gamma}(f)=b_f\in\bar f(\vec a)$.

In particular, $\delta_\gamma(x_i)\in \bar x_i(\vec a)=a_i$ for all $i\le n$. Here $a_i$ denotes the $i$-th coordinate of the vector $\vec a=(a_1,\dots,a_n)$. Also for any constant function $c\in\F$ we get $\delta_\gamma(c)=\bar c(\vec a)=c$. In particular, $\delta_\gamma(0)=0$ and $\delta_\gamma(1)=1$.

To show that $\delta_\gamma$ is an $\mathbb R$-place on the field $\F$, it remains to check that
$\delta_\gamma(f+g)\in\delta_\gamma(f)\oplus \delta_\gamma(g)$ and $\delta_\gamma(f\cdot g)=\delta_\gamma(f)\odot\delta_\gamma(g)$ for any rational functions $f,g\in\F$. Consider the finite subfamily $\EE=\{f,g,f+g,f\cdot g\}\subset\F$, its graph $$\Gamma(\EE)=\{\big(\vec x,(y_e)_{e\in\EE}\big)\in \dom(\EE)\times\IR^\EE:\forall e\in\EE\;\;y_e=e(\vec x)\}$$ and its graphoid $\GGamma(\EE)=\overline{\Gamma(\EE)}\subset \bar\IR^n\times\bar\IR^\EE$. Observe that for any point $\big(\vec x,(y_e)_{e\in\EE}\big)\in\Gamma(\EE)$ we get $$y_{f+g}=(f+g)(\vec x)=f(\vec x)+g(\vec x)=y_f+y_g$$and similarly $y_{f\cdot g}=y_f\cdot y_g$.

Consequently, $\Gamma(\EE)\subset\IR^n\times Y$ where $Y=\{(y_e)_{e\in\EE}\in\IR^\EE:y_{f+g}=y_f+y_g,\;y_{f\cdot g}=y_f\cdot y_g\}$. Observe that the closure of the set $Y$ in $\bar\IR^\EE$ coincides with the subset
$$\bar Y=\{(y_e)_{e\in\EE}\in\bar\IR^\EE: y_{f+g}\in y_f\oplus y_g,\;y_{f\cdot g}=y_f\odot y_g\}.$$
Consequently, $\pr_\EE(\GGamma(\F))\subset\GGamma(\EE)\subset \bar\IR^n\times\bar Y$ which implies the desired inclusions $$\delta_\gamma(f+g)=b_{f+g}\in b_f\oplus b_g=\delta_\gamma(f)\oplus\delta_\gamma(g)$$ and
$$\delta_\gamma(f\cdot g)=b_{f\cdot g}\in b_f\odot b_g=\delta_\gamma(f)\odot\delta_\gamma(g).$$
\smallskip

2. It is easy to see that the map $\delta:\GGamma(\F)\to M(\F)$, $\delta:\gamma\mapsto \delta_\gamma$, is continuous. Let us show that it is injective. Take two distinct points $\gamma=(\vec a,(b_f)_{f\in\F})$ and $\gamma'=(\vec a',(b'_f)_{f\in\F})$ in the graphoid $\GGamma(\F)$. Then either $b_f\ne b'_f$ for some $f\in\F$ or $a_i\ne a_i'$ for some $i\le n$.

If $b_f\ne b'_f$ for some $f$, then $\delta_\gamma(f)=b_f\ne b_f'=\delta_{\gamma'}(f)$ and hence $\delta_\gamma\ne\delta_{\gamma'}$. If $a_i\ne a_i'$ for some $i\le n$, then for the monomial $x_i$, we get $\delta_\gamma(x_i)=x_i(\vec a)=a_i\ne a_i'=x_i(\vec a')=\delta_{\gamma'}(x_i)$ and again $\delta_\gamma\ne\delta_{\gamma'}$.

Therefore, the continuous map $\delta:\GGamma(\F)\to M(\F)$ is injective. Since the space $\GGamma(\F)$ is compact and $M(\F)$ is Hausdorff, the map $\delta$ is a topological embedding.
\smallskip

3. Assume that $\F=\IK(\vec x)$ for some subfield $\IK$ of $\IR$. Then inclusion  $\delta(\GGamma(\F))\subset\{\chi\in M(\F):\chi|\IK=\id\}$ follows from the statement (1).
To prove the reverse inclusion we shall apply the Tarski-Seidenberg Transfer Principle \cite{TTP}. This Principle says that for two real closed extensions $R_1$, $R_2$ of an ordered field $K$, a finite system of inequalities between polynomials with coefficients in $K$ has a solution in $R_1$ if and only if it has a solution in the field $R_2$.

Fix an $\mathbb R$-place $\chi:\F\to\bar\IR$ such that $\chi|\IK=\id$. By \cite{Brown}, the $\mathbb R$-place $\chi$ is induced by some total ordering $P$ of the field $\F$. Taking into account that $\chi|\IK=\id$ is the identity $\mathbb R$-place on the field $\IK$, we conclude that the orders on $\IK$ induced from the ordered fields $(\F,P)$ and $(\IR,\IR_+)$ coincide. Let $\hat\IK$ be the relative  algebraic closure of $\IK$ in the real closed field $\IR$ and $\hat\F$ be a real closure of the ordered field $(\F,P)$. The Uniqueness Theorem \cite[XI.\S2]{Lang} for real closures guarantees that $\hat\IK$ can be identified with the real  closure of $\IK$ in the field $\hat\F$.  By Theorem 6 of \cite{Lang53}, the $\mathbb R$-place $\chi$ extends to a unique $\mathbb R$-place $\hat\chi:\hat\F\to\bar\IR$. The $\mathbb R$-place $\hat\chi|\hat\IK$, being a unique $\mathbb R$-place on the real closed field $\hat\IK$, coincides with the indentity $\mathbb R$-place $\id:\hat\IK\to \IR$.

For every $i\le n$ let $a_i=\chi(x_i)$, $\vec a=(a_1,\dots,a_n)$, and $b_f=\chi(f)$ for $f\in\F$. The inclusion $\chi\in\delta(\GGamma(\F))$ will be proved as soon as we check that the point $\gamma=\big(\vec a,(b_f)_{f\in\F}\big)\in \bar\IR^n\times\bar\IR^\F$ belongs to the graphoid $\GGamma(\F)$. This will follow from Lemma~\ref{l3.1} as soon as for any finite subfamily $\EE\subset\F$, a neighborhood $O(\vec a)\subset\bar\IR^n$ of the point $\vec a=(a_1,\dots,a_n)$ and neighborhoods $O(b_f)\subset \bar\IR$ of the points $b_f$, $f\in\EE$, we find a vector $\vec z=(z_1,\dots,z_n)\in O(\vec a)\cap\dom(\EE)$ such that $f(\vec z)\in O(b_f)$ for all $f\in\EE$.

We loose no generality assuming that $\{x_1,\dots,x_n\}\subset\EE$ and $\prod_{i=1}^nO(\chi(x_i))\subset O(\vec a)$.

Also we can assume that for each function $f\in\EE$ the neighborhood $O(b_f)$ is of the form \begin{itemize}
\item ${]\alpha_f,\beta_f[}$ for some rational numbers $\alpha_f<\beta_f$ if $b_f\in\IR$, and
\item $\bar\IR\setminus [\alpha_f,\beta_f]$ for some rational numbers $\alpha_f<\beta_f$ if $b_f=\infty$.
\end{itemize}

Write each rational function $f\in\EE$ as an irreducible fraction $f=\frac{p_f}{q_f}$ of two polynomials $p_f,q_f\in\IK(\vec x)$. Replacing the polynomials $p_f$ and $q_f$ by $-p_f$ and $-q_f$, if necessary, we can assume that $q_f>0$ in the ordered field $\hat\F$.

Write the finite set $\EE$ as the union $\EE=\EE_-\cup\EE_0\cup\EE_+$ where
$$
\begin{aligned}
\EE_0&=\{f\in\EE:\chi(f)\in\IR\},\\
\EE_-&=\{f\in\EE:\chi(f)=\infty,\;f<0\mbox{ in $\hat\F$}\},\\
\EE_+&=\{f\in\EE:\chi(f)=\infty,\;f>0\mbox{ in $\hat\F$}\}.
\end{aligned}
$$

To each $f\in \EE_0$ we shall assign a system of two polynomial inequalities that has a solution
in the field $\hat\F$. Observe that the inclusion $b_f\in O(b_f)={]\alpha_f,\beta_f[}$ implies that
$\alpha_f<\chi(\frac{p_f}{q_f})=\hat\chi(\frac{p_f}{q_f})<\beta_f$. Since the $\mathbb R$-place $\hat\chi$ is generated by the total order of the real closed field $\hat\F$, these inequalities are equivalent to the inequalities $\alpha_f<\frac{p_f}{q_f}< \beta_f$ holding in the ordered field $\hat\F$. Since $q_f>0$, the latter inequalities are equivalent to $\alpha_fq_f<p_f<\beta_fq_f$. It follows that the vector $\vec x=(x_1,\dots,x_n)\in\hat\F^n$ is a solution of the system
$$\alpha_fq_f(\vec x)<p_f(\vec x)<\beta_fq_f(\vec x)$$in the real closed field $\hat\F$.

Next, consider the case of a function $f\in\EE_+$. Since $\hat\chi(\frac{p_f}{q_f})=\chi(f)=\infty$ and $f>0$, we get $\beta_f q_f<p_f$ in $\hat\F$ and hence the inequality
$$\beta_fq_f(\vec x)<p_f(\vec x)$$ has solution $\vec x=(x_1,\dots,x_n)$ in $\hat\F$.
By the same reason, for every $f\in\EE_-$ the inequality
$$p_f(\vec x)<\alpha_fq_f(\vec x)$$ has solution in $\hat\F$.

Therefore, the system of the inequalities
$$\begin{cases}
q_f(\vec x)>0&\mbox{for all $f\in\EE$}\\
\alpha_fq_f(\vec x)<p_f(\vec x)<\beta_fq_f(\vec x)&\mbox{for all $f\in\EE_0$}\\
\beta_fq_f(\vec x)<p_f(\vec x)&\mbox{for all $f\in\EE_+$}\\
p_f(\vec x)<\alpha_fq_f(\vec x)&\mbox{for all $f\in\EE_-$}
\end{cases}
$$
has solution $\vec x=(x_1,\dots,x_n)$ in the real closed field $\hat\F$. By the Tarski-Seidenberg Transfer Principle \cite[11.2.2]{TTP}, this system has a solution in the real closed field $\hat\IK\subset\IR$. Using the continuity of the polynomials $p_f,q_f$, $f\in\EE$, we can find a solution $\vec z$ of this system in the dense subset $(\hat\IK\cap \dom(\EE))^n$ of $\hat\IK^n$. The choice of the inequalities from the system guarantees that $\vec z\in \prod_{i=1}^nO(\chi(x_i))\cap \dom(\EE)^n\subset O(a)\cap \dom(\EE)^n$ and $f(\vec z)=\frac{p_f(\vec z)}{q_f(\vec z)}\in O(b_f)$ for all $f\in\EE$.
\end{proof}

Theorem~\ref{t3.2} will help us to analyze the structure of certain fibers of the restriction operator $\rho_K:M(K(\vec x))\to M(K),\;\;\rho_K:\chi\mapsto\chi|K$.

\begin{proposition}\label{p3.3} Take  any field $K$ with an injective $\mathbb R$-place $\varphi:K\to\IR$.
Then the fiber $\rho_K^{-1}(\varphi)\subset M(K(\vec x))$
%of the restriction operator $\rho_K:M(K(\vec x))\to M(K)$
can be identified with the graphoid $\GGamma(\F)$, where  $\F = \IK(\vec x)$ for $\IK = \varphi(K) \subset\IR$.
\end{proposition}

\begin{proof} The $\mathbb R$-place $\varphi:K\to\IR$, being injective, is an isomorphism of the fields $K$ and $\IK$. This isomorphism extends to a unique isomorphism $\Phi:K(\vec x)\to\IK(\vec x)$ such that $\Phi(x_i)=x_i$ for all $i\le n$, where $\vec x=(x_1,\dots,x_n)$.

The isomorphism $\varphi:K\to\IK$ induces a homeomorphism $M\varphi:M(\IK)\to M(K)$ which assigns to each $\mathbb R$-place $\chi:\IK\to\bar\IR$ the $\mathbb R$-place $\chi\circ\varphi:K\to\bar\IR$. In the same way the isomorphism $\Phi$ induces a homeomorphism $M\Phi:M(\IK(\vec x))\to M(K(\vec x))$.
Now look at the commutative diagram
$$\xymatrix{
M(K(\vec x))\ar[d]_{\rho_K} &M(\IK(\vec x))\ar[l]_{M\Phi}\ar[d]^{\rho_{\IK}}\\
M(K)&M(\IK)\ar[l]^{M\varphi}& \GGamma(\IK(\vec x))\ar_{\delta}[lu]\ar[ld]\\
\{\varphi\}\ar[u]&\{\id\}\ar[l]\ar[u]
}$$
Here $\delta:\GGamma(\IK(\vec x))\to M(\IK(\vec x))$ is the embedding defined in Theorem~\ref{t3.2}, which implies that $\rho_\IK^{-1}(\id)=\delta(\GGamma(\IK(\vec x)))$. Since the maps $M\varphi$ and $M\Phi$ are homeomorphisms, we conclude that the composition $M\Phi\circ\delta$ maps homeomorphically the graphoid $\GGamma(\IK(\vec x))$ onto the fiber  $\rho_K^{-1}(\varphi)$.
\end{proof}

\section{Extension dimension of the space $M(K(\vec x))$}\label{s4}

In this section we shall evaluate the extension dimension of the space of $\mathbb R$-places
$M(K(\vec x))$ of the field $K(\vec x)=K(x_1,\dots,x_n)$ of rational functions of $n$ variables with coefficients in a field $K$.

 We shall say that a topological space $Y$ is an {\em absolute neighborhood extensor for compacta} (briefly, an ANE) if each continuous map $f:B\to Y$ defined on a closed subspace $B$ of a compact Hausdorff space $X$ can be extended to a continuous map $\bar f:A\to Y$ defined on a neighborhood $A$ of $B$ in $X$.

We recall that a topological space $X$ has extension dimension $\edim X\le Y$ if each continuous map $f:B\to Y$ defined on a closed subspace $B$ of $X$ admits a continuous extension $\bar f:X\to Y$.

\begin{theorem}\label{t4.1} For a  (totally Archimedean) field $K$ the space of $\mathbb R$-places of the field $K(\vec x)$ has extension dimension $\edim M(K(\vec x))\le Y$ for some ANE-space $Y$ (if and) only if for each isomorphic copy $\IK\subset\IR$ of the field $K$ the graphoid $\GGamma(\IK(\vec x))$ has extension dimension $\edim \GGamma(\IK(\vec x))\le Y$. \end{theorem}

\begin{proof}  To prove the ``only if'' part, assume that $\edim M(K(\vec x))\le Y$ for some space $Y$. Given any subfield $\IK\subset\IR$, isomorphic to $K$, we need to check that $\edim\GGamma(\IK(\vec x))\le Y$. Fix any isomorphism $\varphi:K\to\IK$ and observe that it is an injective $\mathbb R$-place on $K$. By Proposition~\ref{p3.3}, the graphoid $\GGamma(\IK(\vec x))$ of the function family $\IK(\vec x)$  is homeomorphic to a subspace of the space $M(K(\vec x))$. Because of that $\edim M(K(\vec x))\le Y$ implies $\edim \GGamma(\IK(\vec x))\le Y$.

The ``if'' part holds under the assumption that the field $K$ is totally Archimedean and the space $Y$ is an ANE. Assume that for each isomorphic copy $\IK\subset\IR$ of the field $K$ the graphoid $\GGamma(\IK(\vec x))$ has extension dimension $\edim\GGamma(\IK(\vec x))\le Y$.  Since the field $K$ is totally Archimedean, each $\mathbb R$-place $\chi:K\to\bar\IR$ is injective, has image in $\IR$ and is generated by a unique total order on $X$ (defined as $x<y$ iff $\chi(x)<\chi(y)$). This means that the quotient map $\lambda:\mathcal X(K)\to M(K)$ is injective and hence is a homeomorphism. Since the space $\mathcal X(K)$ is zero-dimensional, so is the space $M(K)=M_A(X)$.

Now consider the restriction operator $\rho_K:M(K(\vec x))\to M(K)$, $\rho_K:\chi\mapsto \chi|K$.
By Proposition~\ref{p3.3}, for each $\mathbb R$-place $\varphi\in M_A(K)=M(K)$ the fiber $\rho_K^{-1}(\varphi)$ is homeomorphic to the graphoid $\GGamma(\IK(\vec x))$ of the family $\IK(\vec x)$ of rational functions of $n$ variables with coefficients in the subfield $\IK=\varphi(K)$ of $\IR$. Our assumption on the extension dimension of $\GGamma(\IK(\vec x))$ implies that $\edim \rho_K^{-1}(\varphi)\le Y$. The following lemma
implies that $\edim M(K(\vec x))\le Y$.
\end{proof}

\begin{lemma} Let $\rho:X\to Z$ be a continuous map from a compact Hausdorff space $X$ onto a zero-dimensional compact Hausdorff space $Z$. The space $X$ has extension dimension $\edim X\le Y$ for some ANE-space $Y$ if and only if for each $z\in Z$ the fiber $\rho^{-1}(z)$ has extension dimension $\edim f^{-1}(z)\le Y$.
\end{lemma}

\begin{proof} The ``only if'' trivially follows from the definition of extension dimension. To prove the ``if'' part, assume that each fiber of $\rho$ has extension dimension $\le Y$. To prove that $\edim X\le Y$, fix a continuous map $f:B\to Y$ defined on a closed subspace $B$ of $X$. For each point $z\in Z$ consider the fiber $X_z=\rho^{-1}(z)\subset X$ of the map $\rho$. Since $\edim X_z\le Y$, the map $f|B\cap X_z$ admits a continuous extension $f_z:X_z\to Y$. Consider the map $\tilde f_z:X_z\cup B\to Y$ defined by $\tilde f_z|X_z=f_z$ and $\tilde f_z|B=f$. Since $Y$ is an ANE-space, the map $\tilde f_z$ admits a continuous extension $\bar f_z:A_z\to Y$ defined on an open neighborhood $A_z$ of $X_z\cup B$ in $X$. Since the space $X$ is compact and $Z$ is Hausdorff, the map $\rho$ is closed. Consequently, the set $f(X\setminus A_z)$ is closed in $Z$ and its complement $O_z=Z\setminus f(X\setminus A_z)$ is an open neighborhood of $z$ in $Z$. Since the space $Z$ is compact and zero-dimensional, the open cover $\{O_z:z\in Z\}$ of $Z$ can be refined by a finite disjoint open cover $\U$. For every set $U\in\U$ choose a point $z\in Z$ with $U\subset O_z$ and put $\bar f_U=\bar f_z|\rho^{-1}(U)$. It follows that the map $f_U$ is a continuous extension of the map $f|B\cap \rho^{-1}(U)$. Then the maps $\bar f_U$, $U\in\U$, compose a required continuous extension $\bar f=\bigcup_{U\in\U}\bar f_U:X\to Y$ of the map $f$.
\end{proof}

By Hurewicz-Wallman Theorem \cite[1.9.3]{End}, a compact Hausdorff space $X$ has covering topological dimension $\dim X\le d$ for some $d\in\w$ if and only if $\edim X\le S^d$ where $S^d$ stands for the $d$-dimensional sphere. Because of that Theorem~\ref{t4.1} implies:

\begin{corollary}\label{c4.2}  For a (totally Archimedean) field $K$ the space of $\mathbb R$-places of the field $K(\vec x)$ has dimension $\dim M(K(\vec x))\le d$ for some $d\in\w$  (if and) only if for each isomorphic copy $\IK\subset\IR$ of the field $K$ the graphoid $\GGamma(\IK(\vec x))$ has dimension $\dim \GGamma(\IK(\vec x))\le d$.
\end{corollary}

Let us recall \cite{Dran} for an Abelian group $G$ a compact Hausdorff space $X$ has cohomological dimension $\dim_G X\le d$ for some $d\in\w$ if and only if $\edim X\le K(G,d)$ (see Introduction). This fact combined with Theorem~\ref{t4.1} implies:

\begin{corollary}\label{c4.3}  For a (totally Archimedean) field $K$ the space of $\mathbb R$-places of the field $K(\vec x)$ has dimension $\dim_G M(K(\vec x))\le d$ for some $d\in\w$ and some Abelian group $G$  (if and) only if for each isomorphic copy $\IK\subset\IR$ of the field $K$ the graphoid $\GGamma(\IK(\vec x))$ has dimension $\dim_G \GGamma(\IK(\vec x))\le d$.
\end{corollary}

Now we see that Theorems~\ref{t1} and \ref{t2} follows from Corollaries~\ref{c4.2}, \ref{c4.3} and the following deep result \cite{BP} about the (cohomological) dimension of the graphoids.

\begin{theorem}[Banakh-Potyatynyk] \begin{enumerate}
\item For any subfamily $\F\subset\IR(x)$ the graphoid $\GGamma(\F)$ is homeomorphic to the extended real line $\bar\IR$ and hence has dimension $\dim \GGamma(\F)=1$.
\item For any subfamily $\F\subset\IR(x,y)$ the graphoid $\GGamma(\F)$ has dimensions $\dim\GGamma(\F)=\dim_\IZ\GGamma(\F)=2$.
\item For any subfamily $\F\subset\IR(x,y)$  containing the rational functions $\frac{x-a}{y-b}$, $a,b\in\IQ$, the graphoid $\GGamma(\F)$ has cohomological dimension $\dim_G\GGamma(\F)=1$ for any non-trivial 2-divisible Abelian group $G$.
\end{enumerate}
\end{theorem}

In light of Corollaries~\ref{c4.2} and \ref{c4.3} the following problem arises naturally.

\begin{problem} Let $\IK,\IF\subset\IR$ be two isomorphic copies of a (totally Archimedean) field $K$. Are the graphoids $\GGamma(\IK(x,y))$ and $\GGamma(\IF(x,y))$ homeomorphic?
\end{problem}

\begin{remark} \label{AS} In light of this question it is interesting to remark that a totally Archimedean
field can have distinct isomorphic copies in $\IR$. A suitable example can be constructed as follows.
Take the polynomial $f(x)=x^4-5x^2+2$. This polynomial is irreducible over $\IQ$ and has four real roots.
The Galois group of $f$ is the dihedral group with 8 elements, so the degree of the splitting field of $f$ over $\IQ$ is 8.
Therefore, for every root $\alpha$ of $f$ there is another root $\beta$ such that $\beta \notin \IQ(\alpha)$. It follows that
$\IQ(\alpha)$ and $\IQ(\beta)$ are isomorphic, but not equal, totally Archimedean subfields of $\IR$.
\end{remark}

\section{Acknowledgements}

The authors would like to thank Andrzej S\l adek for providing the example in Remark \ref{AS}.

\end{document}